\documentclass[12pt]{article}

\usepackage[T1]{fontenc}
\DeclareFontFamily{T1}{calligra}{}
\DeclareFontShape{T1}{calligra}{m}{n}{<->s*[1.44]callig15}{}
\DeclareMathAlphabet\mathrsfso      {U}{rsfso}{m}{n}

\usepackage{amsmath, hyperref, amsfonts, amssymb, amstext, amsthm}
\usepackage[a4paper,total={6in,9in}]{geometry}
\usepackage{wrapfig}
\usepackage{latexsym,amsmath,amssymb}
\usepackage{graphics,epstopdf}
\usepackage[dvipsnames]{xcolor}
\usepackage{tikz}
\usepackage{parskip}
\usepackage{subcaption}
\usepackage[export]{adjustbox}

\newtheorem{definition}{Definition}
\newtheorem{theorem}{Theorem}

\newtheorem{lemma}{Lemma}
\numberwithin{equation}{section}
\numberwithin{equation}{section}

\begin{document}
	\begin{center}
		{\Large \textbf{Boundary Interpolation on Triangles via Neural Network Operators}}
		
		\bigskip
		
		\textbf{Aaqib Ayoub Bhat}$^{1}$,
		\textbf{Asif Khan}$^{1}$
		\bigskip
		
		$^{1}$Department of Mathematics, Aligarh Muslim University, Aligarh 202002, India\\[0pt]
		bhataqib19@gmail.com; akhan.mm@amu.ac.in;
		
		\bigskip
		
	\end{center}	
	\textbf{Abstract.}
The primary objective of this study is to develop novel interpolation operators that interpolate the boundary values of a function defined on a triangle. This is accomplished by constructing New Generalized Boolean sum neural network operator $\mathcal{B}_{n_1, n_2, \xi }$ using a class of activation functions. Its interpolation properties are established and the estimates for the error of approximation corresponding to operator $\mathcal{B}_{n_1, n_2, \xi }$ is computed in terms of mixed modulus of continuity. The advantage of our method is that it does not require training the network. Instead, the number of hidden neurons adjusts the weights and bias.  Numerical examples are illustrated to show the efficacy of these newly constructed operators. Further, with the help of MATLAB, comparative and graphical analysis is given to show the validity and efficiency of the results obtained for these operators.
	
	\bigskip
	

	\bigskip

	\section{Introduction}

Neural networks (NNs) initially appeared in the early 1900s to develop a basic human brain model. The fundamental aim was to create a network architecture in layers, with each node representing the behavior of a biological neuron. Neural networks are widely investigated for their applications in disciplines like artificial intelligence and machine learning,  computer vision \cite{article2}, medical sciences for diagnosis of various diseases \cite{tan2, tan1, tan3}, self-driving cars, real-time translation software, image processing, stock market predictions, weather prediction, social media user behavior analysis, etc. Additionally, NNs are useful in Approximation Theory, it is widely known that NNs can serve as universal approximators.

	A brief introduction to Neural Network Operators is as following:\\
	Let $\xi:\mathbb{R}\rightarrow \mathbb{R}$ be the activation function. Feed-forward FNNs with  one hidden layer are mathematically expressed as
	\begin{equation}\label{aa}
			\mathcal{N}_n(x)=\sum_{k=1}^{n}c_{k}\xi(a_{k}\cdot x+b_{k})\quad x\in \mathbb{R}^{d},\quad d\in \mathbb{N}
		\end{equation}
	
\noindent	where $c_{k}$ $\in$ $\mathbb{R}$ , $a_{k} \in \mathbb{R}^d, $ and  $b_{k} \in \mathbb{R}$ for $k=1,2,\cdots , n$ are the coefficients, weights and the threshold values respectively. If $\xi$ is defined
	on $\mathbb{R}$, then $\mathcal{N}$ becomes
	\begin{equation}
			\mathcal{N}_n(x)=\sum_{k=1}^{n}c_{k}\xi(\langle A_{k}\cdot x\rangle +b_{k})
		\end{equation}
	Further $\langle A_{k}\cdot x\rangle$ is the inner product of $A_{k}$ and $x$.\\
	 A FNN can approximate every continuous function defined on a compact set to any preassigned degree of accuracy. This can be done by increasing the number of hidden layers \cite{Bar}. Cardaliaguet and Euvrard introduced NN operators in \cite{CARDALIAGUET1992207}, primarily on bell-shaped activation functions having compact support. Several mathematicians have focused on developing NN-based approximation techniques. Cybenko's \cite{Cybenko1989ApproximationBS} work is frequently cited as a foundation for this topic. Cybenko investigated single layer NNs and proved the corresponding approximation theorem using non-constructive approach, such as the Hahn-Banach theorem of functional analysis.
	 
	 Barnhill, Birkhoff, and Gordon's  \cite{BARNHILL1973114} introduction of Boolean sum interpolation theory on triangles inspired many mathematicians to formulate various interpolation operators that interpolate a given function defined on a triangular domain on the boundary of triangle. e.g., \cite{Barnhill1974306, KHAN20215909, Nielson_Thomas_Wixom_1979}. This sort of approximation is used in fields such as finite element analysis \cite{marshall1978blending} and computer-aided geometric design \cite{BARNHILL197769}. In 1912 S.N. Bernstein \cite{bernstein1912constructive} gave the landmark constructive proof of Weierstrass approximation theorem by constructing Bernstein polynomials which lead to the development of constructive approximation theory. P. Blaga et al \cite{Blaga_Coman_2009} extended Bernstein's result on a triangular domain by constructing Bernstein-type operators which interpolate the given target function on the edges of the triangle. For some recent advances on Bernstein type operators on triangular domains and their applications, one is accessible to look at \cite{cai2023approximation, khan2021phillips}.
	 
	 Inspired by Bernstein's constructive proof, NN operators encountered constructive approach and is often implemented by researchers, for example \cite{COSTARELLI201528, COSTARELLI201480, KADAK2022114426, MHASKAR1995151}. The primary objective of this paper is to develop the neural network operators which interpolate the target function on the edges of the triangular domain.

	\section{Preliminaries and known results}
	\begin{definition}
		A measurable function $\xi :\mathbb{R}\rightarrow \mathbb{R}$ is said to be sigmoidal  function if $\xi$ satisfies the following:\\
		$\displaystyle\lim_{x\rightarrow \infty}\xi(x)=1,$ and $\displaystyle\lim_{x\rightarrow -\infty}\xi(x)=0.$ 
	\end{definition}
	
	\noindent Qian et al. \cite{QIAN2022126781} introduced a newly class of sigmodal functions $\mathrsfso{A}(m)$ defined as follows:\\
	
	\begin{definition}
		For fixed $m\in \mathbb{R}^+$, we say a sigmoidal function $\xi \in \mathrsfso{A}(m)$ if it satisfies the following two conditions:\\
		\begin{enumerate}
			\item[(i)] $\xi(x)$ is nondecreasing;
			\item[(ii)] $\xi(x)=1$ for $x\geq m$, and $\xi(x)=0$ for $x\leq -m.$
		\end{enumerate}
	\end{definition}
	
	Now, we take a kind of combinations of the translations of $\xi \in \mathrsfso{A}(m)$ and define a new function $\Psi$ as follows:
	\begin{equation}
		\Psi(x)=\xi(x+m)-\xi(x-m).
	\end{equation}
	
	$\Psi$ satisfies many basic properties given in following lemma.\\
	
	\begin{lemma} \label{l1}
		For $\xi \in \mathrsfso{A}(m)$, $\Psi$ satisfies the following properties:
		\begin{enumerate}
			\item[P1:] $\Psi$ is non-decreasing; 
			\item[P2:] $\Psi$ is non-decreasing for $x\leq 0$, and non-increasing for $x\geq 0$;
			\item [P3:]$supp(\Psi)$ $\subseteq [-2m,2m]$
			\item [P4:] $\Psi(x)+\Psi(x-2m)=1,\quad x\in [0,2m]$
		\end{enumerate}
	\end{lemma}
	
	In sections \ref{s3} and \ref{s4}, we have constructed new parametric extensions  $\mathcal{S}_{n_1,\xi}^{x}$, $\mathcal{S}_{n_2,\xi}^{y}$ and a bivariate extension $\mathcal{P}_{n_1, n_2, \xi}$ of the univariate neural network operator (\ref{operator}) on triangular domain which interpolate the given real valued function on the edges of the triangle. We estimated the error of approximation attained by these neural network operators in terms of modulus of continuity. In section \ref{s5}, we established the neural network version of the Generalized Boolean Sum operators for B\"{o}gel continuous (B-continuous) functions and established the estimates for the error of approximation in terms of mixed modulus of continuity. In sections \ref{s6} and \ref{s7}, we examined our theoretical findings using examples and computed the approximation error for the target function.
	
	\section{Construction of operators} \label{s3}
	We recall a neural network operator  with single hidden layer  introduced by Qian et al. \cite{QIAN2022126781}. Let $\Psi \in \mathrsfso{A}(m)$ and $f:[a,b]\rightarrow \mathbb{R}$ be a bounded and measurable function, the operator $\mathcal{S}_{n, \xi}$ is defined as
	\begin{equation} \label{operator}
		\mathcal{S}_{n,\xi}(f,x) = \sum_{k=0}^{n}f(x_k)\Psi\left(\frac{2m}{h}(x-x_k)\right),
	\end{equation}
	$ \text{where} ~n \in \mathbb{N}^+,~  x_k= a+kh, ~ k=0,1,\cdots n,~\text{with}~ h=\frac{b-a}{n}.$\\
	Now we introduce the parametric extensions $\mathcal{S}_{n_1,\xi}^{x}$ and $\mathcal{S}_{n_2,\xi}^{y}$ of univariate neural network operator $\mathcal{S}_{n,\xi}$ on a triangular domain; i.e., $\mathcal{S}_{n_1,\xi}^{x}$ is defined in such a way as if $\mathcal{S}_{n_1,\xi}$ were applied to a bivariate function $\mathrsfso{F}$ with the second variable fixed and $\mathcal{S}_{n_2,\xi}^{y}$ defined in the similar way.\\
	\indent	Due to the fact that triangle is affine invariant, we consider a standard triangular domain $\Delta$ as follows;
	\begin{equation*}
		\Delta=\{(x,y)\in \mathbb{R}^2:x\geq0, y \geq 0, x+y\leq a, a > 0 \in \mathbb{R} \}
	\end{equation*}
	\begin{wrapfigure}{r}{0.48\textwidth }
		\centering
		\begin{tikzpicture}
		\draw [line width=0.3mm][violet] (0,0) -- (0,4);
		
		\draw [line width=0.3mm][violet] (0,0) -- (4,0);
		
		\draw [line width=0.3mm][violet] (4,0) -- (0,4);
		
		\filldraw (0,0) circle (1pt) node[below left] {(0,0)};
		
		\filldraw (0,4) circle (1pt) node[above] {(0,a)};
		
		\filldraw (4,0) circle (1pt) node[below right] {(a,0)};
		
		\draw [line width=0.3mm][violet] (0,1.3) -- (2.7,1.3);
		
		\draw [line width=0.3mm][violet] (1.3,0) -- (1.3,2.7);
		\filldraw [RedViolet] (1.4,0)   node[below]{$C_x$};
		
		\filldraw [RedViolet] (0, 1.4)   node[left]{$A_y$};
		
		\filldraw [RedViolet] (2.6, 1.5)   node[right]{$B_y$};
		
		\filldraw [RedViolet] (1.5,2.6)   node[above]{$D_x$};
		
		\filldraw [RedViolet] (-0.1,2)   node[left]{$\Gamma_2$};
		
		\filldraw [RedViolet] (2.2, -0.1)   node[below]{$\Gamma_1$};
		
		\filldraw [RedViolet] (2, 2)   node[above right]{$\Gamma_3$};
		\end{tikzpicture}
		\caption{  }\label{f11}
	\end{wrapfigure}
	Let
	\begin{align*}
		\varGamma_{1}&=\{(x,y)\in  \Delta: x= 0\},\\
		\varGamma_{2}&=\{(x,y)\in  \Delta: y= 0\}, and\\ \varGamma_{3}&=\{(x,y)\in  \Delta: x+y=a\}
	\end{align*}	
	\noindent Notice that $A_{y}=(0, y)$ and  $B_{y}=(a-y, y)$ are the end points of line segment $A_yB_y$ from $\varGamma_{2}$ to $\varGamma_{3}$ parallel to X-axis. Similarly $C_{x}=(x,0)$ and $D_{x}=(x,a-x)$ are the endpoints of line segment $C_xD_x$ from $\varGamma_{1}$ to $\varGamma_{3}$ parallel to Y-axis as shown in figure \ref{f11}.

	\noindent
	For a function  $\mathrsfso{F}:\Delta\rightarrow \mathbb{R}$ we define:  	
	\begin{equation}\label{bmx}
		(\mathcal{S}_{n_1,\xi}^{x}\mathrsfso{F})(x,y) =  \left\{
		\begin{array}{ll}
			\displaystyle\sum_{k=0}^{n_1}\mathrsfso{F}(x_k, y) \Psi\left(\frac{2mn_1}{a-y}(x-x_k)\right) \quad (x,y)\in  \Delta \setminus \{(0,a)\}& \\
			& \\
			\mathrsfso{F}(0,a),~~~~~~~~~~~~~~~~~~~~~~~~~~~~~~~~~~~~~~~~~(0,a) \in  \Delta&
		\end{array}%
		\right.
	\end{equation}
	and
	\begin{equation}\label{bmy}
		(\mathcal{S}_{n_2,\xi}^{y}\mathrsfso{F})(x,y) =  \left\{
		\begin{array}{ll}
		\displaystyle\sum_{k=0}^{n_2}\mathrsfso{F}(x,y_l) \Psi\left(\frac{2mn_2}{a-x}(y-y_l)\right) \quad (x,y)\in  \Delta \setminus \{(a,0)\}  & \\
			& \\
			\mathrsfso{F}(a,0),~~~~~~~~~~~~~~~~~~~~~~~~~~~~~~~~~~~~~~~~~~~~~~(a,0)\in  \Delta,&
		\end{array}%
		\right.
	\end{equation}
	where $n_1, n_2$ $\in \mathbb{N}^+$, $x_k =\frac{k}{n_1}(a-y)$, and $y_l =\frac{l}{n_2}(a-x).$\\
	
	Equation \ref{bmx} describes a neural network operator with single hidden layer in which $\mathrsfso{F}(x_k,y)$ and $\frac{2mn_1}{a-y}$ represents the connection weights between adjoining layers i.e., from input layer to hidden layer and from hidden layer to output layer respectively, and  $\frac{2mn_1}{a-y}x_k$ represents the threshold (bias). Clearly, the connection weights and bias depend on the number of hidden neurons $n_1$. Hence, weights and bias are adjusted by the number of hidden neurons. Similarly, equation \ref{bmy} defines a neural network operator with a single hidden layer.
	

	\begin{lemma}\label{l2}
		For any $(x,y) \in  \Delta $, we have
		\begin{equation}
			(\mathcal{S}_{n_1,\xi}^{x}\textbf{1})(x,y) =1,
		\end{equation}
		where $\textbf{1}(x,y)=1 \quad \forall (x,y) \in  \Delta.$
	\end{lemma}
	Firstly, we show that the operators defined in (\ref{bmx}) interpolate $\mathrsfso{F}$ on $\Gamma_2 \cup \Gamma_3.$\\
	
	\begin{theorem} \label{t1}
		Let $\mathrsfso{F}:\Delta\rightarrow \mathbb{R}$ be a bounded and measurable function, $\xi \in \mathrsfso{A}(m)$. Then
		\begin{equation*}
			(\mathcal{S}_{n_1,\xi}^{x}\mathrsfso{F})=\mathrsfso{F}~on~ \Gamma_2 \cup \Gamma_3.
		\end{equation*}
	\end{theorem}
	\begin{proof}
		Let $(\tilde{x},\tilde{y})$ $\in \Gamma _2$ be arbitrary, we have
		$\tilde{x}=0$. Using the property (P3) in lemma \ref{l1}\\
		\begin{align*}
			(\mathcal{S}_{n_1,\xi}^{x}\mathrsfso{F})(0,\tilde{y})&=\sum_{k=0}^{n_1}\mathrsfso{F}(x_k, \tilde{y}) \Psi\left(\frac{2mn_1}{a-\tilde{y}}(0-x_k)\right), \quad (x,y)\in \Gamma_2 \setminus \{(0,a)\}\\
			&=\sum_{k=0}^{n_1}\mathrsfso{F}(x_k, \tilde{y}) \Psi \left(\frac{2mn_1}{a-\tilde{y}}(-\frac{k}{n_1}(a-\tilde{y}))\right)\\
			&=\sum_{k=0}^{n_1}\mathrsfso{F}(x_k, \tilde{y})\Psi(-2mk)\\
			&=\mathrsfso{F}(x_0, \tilde{y})\\
			&=\mathrsfso{F}(0, \tilde{y})
		\end{align*}
		
	Similarly, we can prove for $(\tilde{x},\tilde{y})$ $\in \Gamma _3$.
	
	\end{proof}
	We have a similar result for the operator $\mathcal{S}_{n_2,\xi}^{y}\mathrsfso{F}.$
	
	\begin{theorem} \label{t2}
		Let $\mathrsfso{F}:\Delta\rightarrow \mathbb{R}$ be a bounded and measurable function, $\xi \in \mathrsfso{A}(m)$. Then
		\begin{equation*}
			(\mathcal{S}_{n_2,\xi}^{y}\mathrsfso{F})=\mathrsfso{F}~on~ \Gamma_1 \cup \Gamma_3.
		\end{equation*}
	\end{theorem}
\noindent	The proof of this theorem runs parallel to the proof of previous theorem.\\

	\begin{definition}
		Let $C[a,b]$ denote the set of all continuous functions defined on compact interval $[a,b]$. Recall the definition of modulus of continuity of function f:
		\begin{equation}
			\omega(f,\delta)=\sup_{x,y \in [a,b],~ |x-y|\leq\delta}|f(x)-f(y)|
		\end{equation} 
		or equivalently,
		\begin{equation}
			\omega(f,\delta)=\sup_{x,y \in [a,b],~ |h|\leq\delta}|f(x+h)-f(x)|.
		\end{equation}	
	\end{definition}
	Now, we will compute the estimate for the error of approximation  corresponding the operator  $\mathcal{S}_{n_1,\xi}^{x}\mathrsfso{F}$.\\
	
	\begin{theorem}
		If $\mathrsfso{F}(\cdot,y)\in C[0,a-y]$, then 
		\begin{equation}\label{1}
			\|	(\mathcal{S}_{n_1,\xi}^{x}\mathrsfso{F})-\mathrsfso{F}\| \leq \omega (\mathrsfso{F}(\cdot,y),h_1), \quad \quad where ~ h_1=\frac{a-y}{n_1}.
		\end{equation}
	\end{theorem}
	\begin{proof}
		Assume that $x\in [x_i,x_{i+1}], i=0,1,2,\cdots n-1$. Using (P3) in lemma \ref{l1}, we have
		\begin{equation}\label{a}
			\Psi\left(\frac{2mn_1}{a-y}(x-x_k)\right)=0, \quad k\neq i, i+1 .
		\end{equation}
		By using lemma \ref{l2} and equation (\ref{a})
		\begin{align*}
			|(\mathcal{S}_{n_1,\xi}^{x}\mathrsfso{F}(x,y))-\mathrsfso{F}(x,y)|&=\left|\sum_{k=0}^{n_1}\left(\mathrsfso{F}(x_k, y) -\mathrsfso{F}(x,y)\right) \Psi\left(\frac{2mn_1}{a-y}(x-x_k)\right) \right| \\
			&\leq \left|\mathrsfso{F}(x_i, y) -\mathrsfso{F}(x,y)\right|\Psi\left(\frac{2mn_1}{a-y}(x-x_i)\right) \\
			&\quad + \left|\mathrsfso{F}(x_{i+1}, y) -\mathrsfso{F}(x,y)\right| \Psi\left(\frac{2mn_1}{a-y}(x-x_{i+1})\right)\\
			&= \omega (\mathrsfso{F}(\cdot,y),h_1).
		\end{align*}
	\end{proof}
	We have a similar result for $\mathrsfso{F}(x,\cdot)\in C[0,a-x]$ as follows:\\
	\begin{theorem}
		If $\mathrsfso{F}(\cdot,y)\in C[0,a-x]$, then 
		\begin{equation}\label{2}
			\|	(\mathcal{S}_{n_1,\xi}^{y}\mathrsfso{F})-\mathrsfso{F}\| \leq \omega (\mathrsfso{F}(x,\cdot),h_2), \quad \text{where}~ h_2=\frac{a-x}{n_2}.
		\end{equation}
	\end{theorem}
\noindent	This theorem can be proved in a similar way as that of previous theorem.

	\section{Bivariate neural network operator}\label{product} \label{s4}
	
	Let $\mathrsfso{F}:\Delta \rightarrow \mathbb{R}$ be any function. For $\xi \in \mathrsfso{A}(m), n_1, n_2 \in \mathbb{N}$, we define a bivariate extension $\mathcal{P}_{n_1,n_2,\xi}\mathrsfso{F}$ of operator $\mathcal{S}_{n, \xi}\mathrsfso{F}$ as:\\
	\begin{align}
		(\mathcal{P}_{n_1,n_2,\xi}\mathrsfso{F})(x,y) &=\sum_{k=0}^{n_1} \sum_{l=0}^{n_2} \mathrsfso{F} \left(x_k,y_l \right) \Psi\left(\frac{2mn_1}{a-y}(x-x_k)\right)  \Psi \left(\frac{2mn_2}{a-x}(y-y_l)\right).
	\end{align}
	\begin{theorem} \label{t5}
		Let $\mathrsfso{F}:\Delta\rightarrow \mathbb{R}$ be a bounded and measurable function, $\xi \in \mathrsfso{A}(m)$. Then
		\begin{enumerate}
			\item[(i)] \label{51}
				$(\mathcal{P}_{n_1,n_2,\xi}\mathrsfso{F})=\mathrsfso{F}~on~ \Gamma_3.$
			
			\item[(ii)] \label{52}
			$	(\mathcal{P}_{n_1,n_2,\xi}\mathrsfso{F})(x,0)= (\mathcal{S}_{n_1,\xi}^{x}\mathrsfso{F})(x,0)$ i.e., $(\mathcal{P}_{n_1,n_2,\xi}\mathrsfso{F})$ and  $(\mathcal{S}_{n_1,\xi}^{x}\mathrsfso{F})$ agree on $\Gamma_1$.
			
			\item[(iii)] \label{53}
				$(\mathcal{P}_{n_1,n_2,\xi}\mathrsfso{F})(0,y)= (\mathcal{S}_{n_2,\xi}^{y}\mathrsfso{F})(0,y) $ i.e., $(\mathcal{P}_{n_1,n_2,\xi}\mathrsfso{F})$ and  $(\mathcal{S}_{n_2,\xi}^{y}\mathrsfso{F})$ agree on $\Gamma_2$.
		\end{enumerate}

	\end{theorem} 
	\begin{proof}
	\begin{enumerate}
		\item [(i)]	Let $(\tilde{x},\tilde{y})$ $\in \Gamma _3$ be arbitrary. Using the property (P3) in lemma \ref{l1}, we have
		\begin{align*}
			(\mathcal{P}_{n_1,n_2,\xi}\mathrsfso{F})(\tilde{x},\tilde{y}) &=\sum_{k=0}^{n_1} \sum_{l=0}^{n_2} \mathrsfso{F} (x_k,y_l) \Psi\left(\frac{2mn_1}{a-\tilde{y}}(\tilde{x}-x_k)\right) \Psi \left(\frac{2mn_2}{a-\tilde{x}}(\tilde{y}-y_l)\right) \\
			&=\sum_{k=0}^{n_1} \sum_{l=0}^{n_2} \mathrsfso{F} (x_k,y_l) \Psi\left(2m(n_1 -k) \right) \Psi \left( 2m(n_2 -l)\right) \\
			&= \mathrsfso{F}(x_{n_1}, y_{n_2})\\
			&= \mathrsfso{F}(\tilde{x}, \tilde{y}).
		\end{align*}
			Hence, we have proved that\\ $(\mathcal{P}_{n_1,n_2,\xi}\mathrsfso{F})(\tilde{x}, \tilde{y})=\mathrsfso{F}(\tilde{x}, \tilde{y})$ on $ \Gamma _3$.\\
			\item [(ii)] Now,
			\begin{align*}
					(\mathcal{P}_{n_1,n_2,\xi}\mathrsfso{F})(x,0) &=\sum_{k=0}^{n_1} \sum_{l=0}^{n_2} \mathrsfso{F} (x_k,y_l) \Psi\left(\frac{2mn_1}{a}(x-x_k)\right) \Psi \left(\frac{2mn_2}{a-x}(0-y_l)\right) \\
				&=\sum_{k=0}^{n_1} \sum_{l=0}^{n_2} \mathrsfso{F} (x_k,y_l) \Psi\left(\frac{2mn_1}{a}\left(x-x_k\right)\right) \Psi \left( 2m(-l)\right) \\
				&=\sum_{k=0}^{n_1} \mathrsfso{F} (x_k,0) \Psi\left(\frac{2mn_1}{a}\left(x-x_k\right)\right) \\
				&= (\mathcal{S}_{n_1,\xi}^{x}\mathrsfso{F})(x,0).
			\end{align*}
	\end{enumerate}
		Similarly, we can show that
		\begin{equation*}
			 (\mathcal{P}_{n_1,n_2,\xi}\mathrsfso{F})(0, y)=(\mathcal{S}_{n_2,\xi}^{y}\mathrsfso{F})(0,y).
		\end{equation*}
	\end{proof}
	
	\begin{lemma}\label{Lemma3}
		For any $(x,y) \in  \Delta $, we have
		\begin{equation}
			\sum_{k=0}^{n_1} \sum_{l=0}^{n_2}\Psi\left(\frac{2mn_1}{a-y}(x-x_k)\right)\Psi\left(\frac{2mn_2}{a-x}(y-y_l)\right)=1.
		\end{equation}
	\end{lemma}
\noindent	This lemma can be proved by using similar arguments as in lemma \ref*{l2}.
	\begin{definition}
		Let $C(\mathbf{X})$ be the set of all continuous functions $f:\mathbf{X}\rightarrow \mathbb{R}$ defined on the compact set $\mathbf{X}$, the modulus of continuity of f is defined by
		\begin{equation}
			\omega(f,\delta_1, \delta_2)=\sup_{(x,y) ~\in \mathbf{X}, ~|h|\leq\delta_1,~ |k|\leq\delta_2}|f(x+h,y+k)-f(x,y)|.
		\end{equation}
	\end{definition}
\noindent Now, we will compute the estimate for the error of approximation  corresponding to the operator  $\mathcal{P}_{n_1, n_2,\xi}\mathrsfso{F}$ in terms of modulus of continuity.\\

	\begin{theorem}
		Let $\mathrsfso{F} \in C(\Delta)$, then we have\\
		\begin{equation}\label{3}
			\|\mathrsfso{F}-(\mathcal{P}_{n_1,n_2,\xi}\mathrsfso{F})\| \leq \omega (\mathrsfso{F},h_1, h_2),
		\end{equation}
		where $h_1 = \frac{a}{n_1}$ and $h_2 = \frac{a}{n_2}.$
	\end{theorem}
	\begin{proof}
		Assume that $x$ lies between $x_i$ and $x_{i+1}$, and $y$ lies between  $y_j$ and $y_{j+1}$. Using lemma \ref*{Lemma3} and  properties P3 and P4 of lemma \ref{l1}, we have 
		\begin{align*}
			&|\mathrsfso{F}(x,y)-(\mathcal{P}_{n_1,n_2,\xi}\mathrsfso{F})(x,y)|\\
			&=\bigg|\sum_{k=0}^{n_1} \sum_{l=0}^{n_2} \left(\mathrsfso{F}(x,y)- \mathrsfso{F}(x_k,y_l)\right)\Psi\left(\frac{2mn_1}{a-y}(x-x_k)\right)  \Psi\left(\frac{2mn_2}{a-x}(y-y_l)\right)\bigg| \\
			&= \sum_{k=i,i+1}\sum_{l=j,j+1} \bigg| \mathrsfso{F}(x,y)- \mathrsfso{F} (x_k,y_l) \bigg|\Psi\left(\frac{2mn_1}{a-y}(x-x_k)\right)  \Psi\left(\frac{2mn_2}{a-x}(y-y_l)\right)\\
			&\leq \omega(\mathrsfso{F},h_1,h_2)\Bigg( \Psi\left(\frac{2mn_1}{a-y}(x-x_{i})\right)  \Psi\left(\frac{2mn_2}{a-x}(y-y_{j})\right)\\
			 &\quad +\Psi\left(\frac{2mn_1}{a-y}(x-x_{i+1})\right)  \Psi\left(\frac{2mn_2}{a-x}(y-y_{j})\right)\\
			 &\quad +\Psi\left(\frac{2mn_1}{a-y}(x-x_{i})\right)  \Psi\left(\frac{2mn_2}{a-x}(y-y_{j+1})\right)\\ &\quad +\Psi\left(\frac{2mn_1}{a-y}(x-x_{i+1})\right)  \Psi\left(\frac{2mn_2}{a-x}(y-y_{j+1})\right) \Bigg)\\
			&=\omega(\mathrsfso{F},h_1,h_2)\Bigg(\Psi\left(\frac{2mn_1}{a-y}(x-x_{i})\right)\left( \Psi\left(\frac{2mn_2}{a-x}(y-y_{j})\right)+\Psi\left(\frac{2mn_2}{a-x}(y-y_{j+1})\right) \right)\\
			& \quad+ \Psi\left(\frac{2mn_1}{a-y}(x-x_{i+1})\right)\left( \Psi\left(\frac{2mn_2}{a-x}(y-y_{j})\right)+\Psi\left(\frac{2mn_2}{a-x}(y-y_{j+1})\right) \right) \Bigg)\\
			&=\omega(\mathrsfso{F},h_1,h_2)\Bigg(\Psi\left(\frac{2mn_1}{a-y}(x-x_{i})\right)+\Psi\left(\frac{2mn_1}{a-y}(x-x_{i+1})\right)\Bigg)\\
			&=\omega(\mathrsfso{F},h_1,h_2).
		\end{align*}
	\end{proof}
	\section{Generalized Boolean sum neural network operator} \label{s5}
	In this section, we shall introduce the neural network version of Generalized Boolean Sum operators for B\"{o}gel continuous (B-continuous) functions.
	In 1934, B\"{o}gel \cite{Bögel1934} introduced the concept of B-continuous and B-differentiable functions. For detailed literature, one can refer to \cite{Bögel1962}, \cite{Bögel1934}, and \cite{Bögel1935}. Recently several authors have contributed in this direction.  Badea and Cottin \cite{badea1990korovkin} proved the Korovkin type theorems for GBS (Generalized Boolean Sum) operators, S.A. Mohiuddine \cite{SAM}  constructed the bivariate form of Bernstein–Schurer type operator and gave the associated GBS operator, Q.B. Cai et al  introduced GBS operators of $q$-Bernstein-Kantorovich Type \cite{cai} and $\lambda$-Bernstein-Kantorovich Type \cite{cai2}, Ansari et al \cite{articlekj} used a summability method and constructed Bernstein-Kantorovich operatorsin two variables and related GBS operators, Khan et al \cite{KHAN20215909} constructed Boolean Sum operators based on quantum analogue.\\
	
	A  brief introduction of B\"{o}gel continuous functions is as follows:
	\begin{definition}
		Let $f:\mathbf{X}\rightarrow \mathbb{R}$ be any function, where $\mathbf{X}$ is a compact subset of $\mathbb{R}^2$, the mixed difference $\Lambda_{(s,t)} f(x,y)$ is defined as:
		\begin{equation*}
			\Lambda_{(s,t)} f(x,y)= f(x,y)-f(s,y)-f(x,t)+f(s,t).
		\end{equation*}
		If $\displaystyle\lim_{(x,y)\rightarrow (s,t)} \Lambda_{(s,t)}f(x,y)=0$, then $f$ is said to be B\"{o}gel continuous on the compact domain $\mathbf{X}$.
	\end{definition}
\noindent	Denote the space of all B\"{o}gel continuous functions defined on compact domain $\mathbf{X}$ by $C_{B}(\mathbf{X})$.\\
	Now, we shall introduce the Generalized Boolean Sum Neural Network Operator as follows:\\
	Let $\mathrsfso{F} \in C_{B}(\Delta)$, the Generalized Boolean Sum Neural Network Operator $\mathcal{B}_{n_1, n_2, \xi}$ is defined as
	\begin{align*}
		\mathcal{B}_{n_1, n_2, \xi}=\mathcal{S}_{n_1,\xi}^{x}+\mathcal{S}_{n_2,\xi}^{y}-\mathcal{P}_{n_1, n_2, \xi}.
	\end{align*}
	
	\begin{theorem}
		Let  $\mathrsfso{F} \in C_{B}(\Delta)$, $\xi \in \mathrsfso{A}(m)$. Then
		\begin{equation*}
			(\mathcal{B}_{n_1, n_2, \xi}\mathrsfso{F})(\tilde{x},\tilde{y})=\mathrsfso{F}(\tilde{x},\tilde{y})~~~~\forall~~ (\tilde{x},\tilde{y}) \in  \partial\Delta.
		\end{equation*}
		
	\end{theorem}
	\begin{proof}
		Since
		\begin{equation*}
			\mathcal{B}_{n_1,n_2}\mathrsfso{F}=(\mathcal{S}_{n_1,\xi}^{x}+\mathcal{S}_{n_2,\xi}^{y}-\mathcal{P}_{n_1, n_2, \xi})\mathrsfso{F}
		\end{equation*}
Let $(x,y)$ $\in \Gamma _1$ be arbitrary, we have
$y=0$. Using Theorem \ref{52} (ii) and Theorem \ref{t2}		
\begin{align*}
	(\mathcal{B}_{n_1,n_2}\mathrsfso{F})(x,0)&=(\mathcal{S}_{n_1,\xi}^{x}\mathrsfso{F})(x,0)+(\mathcal{S}_{n_2,\xi}^{y}\mathrsfso{F})(x,0)-(\mathcal{P}_{n_1, n_2, \xi}\mathrsfso{F})(x,0)\\
	&=(\mathcal{S}_{n_1,\xi}^{x}\mathrsfso{F})(x,0)+(\mathcal{S}_{n_2,\xi}^{y}\mathrsfso{F})(x,0)-(\mathcal{S}_{n_1,\xi}^{x}\mathrsfso{F})(x,0)\\
	&=(\mathcal{S}_{n_2,\xi}^{y}\mathrsfso{F})(x,0)\\
	&=\mathrsfso{F}(x,0).
\end{align*}
Hence we have proved that $\mathcal{B}_{n_1,n_2}\mathrsfso{F}=\mathrsfso{F}\quad \text{on}~ \Gamma_1.$

\noindent Similarly, using Theorem \ref{52} (iii) and Theorem \ref{t1}, we can prove that $\mathcal{B}_{n_1,n_2}\mathrsfso{F}=\mathrsfso{F}~ \text{on}~ \Gamma_2.$

For $(x,y)\in \Gamma _3$, the interpolation properties of $(\mathcal{S}_{n_1,\xi}^{x}\mathrsfso{F}),~(\mathcal{S}_{n_2,\xi}^{y}\mathrsfso{F}),~ and~ (\mathcal{P}_{n_1, n_2, \xi}\mathrsfso{F})$ on $\Gamma_3$ imply that $\mathcal{B}_{n_1,n_2}\mathrsfso{F}=\mathrsfso{F}$ on $\Gamma_3$		
	\end{proof}
	\begin{definition}
		Let $C_{B}(\mathbf{X})$ be the set of all B-continuous functions $f:\mathbf{X}\rightarrow \mathbb{R}$ defined on the compact set $\mathbf{X}\subset\mathbb{R}^2$, the mixed modulus of continuity \cite{Anastassiou2000, GONSKA1990170} of f is defined by
		\begin{equation}
			\omega_{mixed}(f,\delta)=\sup_{(x,y) \in \mathbf{X}, ~|h|\leq\delta,~ |k|\leq\delta}|f(x+h,y)+f(x,y+k)-f(x+h,y+k)-f(x,y)|.
		\end{equation}
	\end{definition}
\noindent	Now, we will compute the estimate for the error of approximation  corresponding to the operator  $\mathcal{B}_{n_1, n_2,\xi}^{x}\mathrsfso{F}$ in terms of mixed modulus of continuity.\\
	\begin{theorem}
		Let  $\mathrsfso{F} \in C_{B}(\Delta)$, $\xi \in \mathrsfso{A}(m)$. Then
		\begin{equation*}
			\|\mathcal{B}_{n_1, n_2, \xi}-\mathrsfso{F}\|\leq \omega_{mixed}(f,h_1,h_2).
		\end{equation*}
	\end{theorem}
	\begin{proof}
		Assume that $x$ lies between $x_i$ and $x_{i+1}$, and $y$ lies between  $y_j$ and $y_{j+1}$. Using lemma (\ref*{Lemma3}) and  properties P3 and P4 of lemma \ref{l1}, we have 
		\begin{align*}
			&|(\mathcal{B}_{n_1, n_2, \xi})(x,y)-\mathrsfso{F}(x,y)|\\
			&\scalebox{0.92}{$=\bigg|\displaystyle\sum_{k=0}^{n_1} \sum_{l=0}^{n_2} \left(\mathrsfso{F}(x_k,y)+\mathrsfso{F}(x,y_l)-\mathrsfso{F}(x_k,y_l)\right)\Psi\left(\frac{2mn_1}{a-y}(x-x_k)\right)\Psi\left(\frac{2mn_2}{a-x}(y-y_l)\right)-\mathrsfso{F}(x,y)\bigg| $}\\
			&\leq\scalebox{0.90}{$\displaystyle\sum_{k=i,i+1}\sum_{l=j,j+1} \bigg|\mathrsfso{F}(x_k,y)+\mathrsfso{F}(x,y_l)-\mathrsfso{F}(x_k,y_l)-\mathrsfso{F}(x,y)\bigg|\Psi\left(\frac{2mn_1}{a-y}(x-x_k)\right)  \Psi\left(\frac{2mn_2}{a-x}(y-y_l)\right)$}\\
			&\scalebox{0.95}{$\leq \omega_{mixed}(\mathrsfso{F},h_1,h_2)\Bigg( \Psi\left(\frac{2mn_1}{a-y}(x-x_{i})\right)  \Psi\left(\frac{2mn_2}{a-x}(y-y_{j})\right)$} \scalebox{0.95}{$+\Psi\left(\frac{2mn_1}{a-y}(x-x_{i+1})\right)  \Psi\left(\frac{mn_2}{a-y}(y-y_{j})\right)$}\\ &\quad \scalebox{0.95}{$+\Psi\left(\frac{2mn_1}{a-y}(x-x_{i})\right)  \Psi\left(\frac{2mn_2}{a-x}(y-y_{j+1})\right)$} +\Psi\left(\frac{2mn_1}{a-y}(x-x_{i+1})\right)  \Psi\left(\frac{2mn_2}{a-x}(y-y_{j+1})\right) \Bigg)\\
			&\scalebox{0.95}{$=\omega_{mixed}(\mathrsfso{F},h_1,h_2)\Bigg(\Psi\left(\frac{2mn_1}{a-y}(x-x_{i})\right)\left( \Psi\left(\frac{2mn_2}{a-x}(y-y_{j})\right)+\Psi\left(\frac{2mn_2}{a-x}(y-y_{j+1})\right) \right)$}\\
			& \quad \scalebox{0.95}{$+ \Psi\left(\frac{2mn_1}{a-y}(x-x_{i})\right)\left( \Psi\left(\frac{2mn_2}{a-x}(y-y_{j})\right)+\Psi\left(\frac{2mn_2}{a-x}(y-y_{j+1})\right) \right) \Bigg)$}\\
			&\scalebox{0.95}{$=\omega_{mixed}(\mathrsfso{F},h_1,h_2)\Bigg(\Psi\left(\frac{2mn_1}{a-y}(x-x_{i})\right)+\Psi\left(\frac{2mn_1}{a-y}(x-x_{i+1})\right)\Bigg)$}\\
			&\scalebox{0.95}{$=\omega_{mixed}(\mathrsfso{F},h_1,h_2).$}\\
		\end{align*}
	\end{proof}
	\section{Numerical Examples and Error Analysis} \label{s6}
In this section, we provide some illustrations to support the applicability of our theoretical conclusions. We consider  $\mathrsfso{F}(x,y)=\sin(10x)+\cos(5y)$  defined on triangle $\Delta$ as the target function.\\
	In \cite{QIAN2022126781}, a few examples of activation function in $\mathrsfso{A}(m)$  were provided. In this paper, the following activation functions are taken into consideration for computational and graphical investigations of these operators:
	
	\begin{enumerate}
		\item A special ramp function $\rho_{R}$, introduced by Yu in \cite{Yu2016623}
		\begin{align}
			\rho_{R}(\varkappa)=  \left\{
			\begin{array}{ll}
				0,\quad &\varkappa\leq -\frac{1}{2}  \\
				10(\varkappa+\frac{1}{2})^{3} -15(\varkappa+\frac{1}{2})^{4}+6(\varkappa+\frac{1}{2})^{5}, \quad & -\frac{1}{2} \leq \varkappa \leq \frac{1}{2} \\ 1, \quad& \varkappa \geq \frac{1}{2}
			\end{array}.
			\right.
		\end{align}
		One can easily verify that $\rho_R \in \mathrsfso{A}(\frac{1}{2})$.\\
		Let $\varphi_{R}(\varkappa)=\rho_{R}(\varkappa+\frac{1}{2})-\rho_{R}(\varkappa-\frac{1}{2})$. Then
		\begin{equation}
			\varphi_{R}(\varkappa) =  \left\{
			\begin{array}{ll}
				0,\quad &\varkappa\leq -1  \\
				10(\varkappa+1)^{3} -15(\varkappa+1)^{4}+6(\varkappa+1)^{5}, \quad & -1 \leq \varkappa \leq 0
				\\1-(10\varkappa^3 -15\varkappa^4 +6\varkappa^5), \quad & 0 \leq \varkappa \leq  1 \\ 0, \quad & \varkappa \geq 1
			\end{array}
			\right.
		\end{equation}

		\item Another ramp function $\sigma_{R}$, introduced by Costarelli in \cite{COSTARELLI2014574}
		\begin{align}
			\sigma_{R}(\varkappa)=  \left\{
			\begin{array}{ll}
				0,\quad &\varkappa\leq -\frac{1}{2}  \\
				\varkappa + \frac{1}{2}, \quad & -\frac{1}{2} \leq \varkappa \leq \frac{1}{2} \\ 1, \quad& \varkappa \geq \frac{1}{2}
			\end{array}
			\right.
		\end{align}
		One can easily verify that $\sigma_R \in \mathrsfso{A}(\frac{1}{2})$.\\
		Let $\Psi_{R}(\varkappa)=\sigma_{R}(\varkappa+\frac{1}{2})-\sigma_{R}(\varkappa-\frac{1}{2})$. Then
		\begin{equation}
			\Psi_{R}(\varkappa) =  \left\{
			\begin{array}{ll}
				0,\quad &\varkappa\leq -1  \\
				\varkappa+1, \quad & -1 \leq \varkappa \leq 0
				\\1-\varkappa, \quad & 0 \leq \varkappa \leq  1 \\ 0, \quad & \varkappa \geq 1
			\end{array}
			\right.
		\end{equation}

		\item 
		\begin{align}
			\sigma(\varkappa) =  \left\{
			\begin{array}{ll}
				0\quad & \varkappa\leq -1   \\
				\frac{1}{2}\varkappa+\frac{1}{2} \quad &-1 \leq \varkappa \leq 1
				\\ 1 \quad &\varkappa \geq 1
			\end{array}
			\right.
		\end{align}
		One can easily verify that $\sigma \in \mathrsfso{A}(1)$.\\
		Let $\varphi(\varkappa)=\sigma(\varkappa+1)-\sigma(\varkappa-1)$. Then
		\begin{equation}
			\varphi(\varkappa) =  \left\{
			\begin{array}{ll}
				0,\quad & \varkappa\leq -2   \\
				-\frac{1}{2}\varkappa+1, \quad & -2 \leq \varkappa \leq 0
				\\ \frac{1}{2}\varkappa+1, \quad & 0 \leq \varkappa \leq 2
				\\0, \quad & \varkappa \geq 2 
			\end{array}.
			\right.
		\end{equation}
	\end{enumerate}
	Table \ref{T1} gives the approximation error by operators $\mathcal{S}_{n_1, \rho_{R}\sigma}^{x}(\mathrsfso{F})$, $\mathcal{S}_{n_2, \rho_{R}}^{y}(\mathrsfso{F})$, $\mathcal{P}_{n_1, n_2, \rho_{R}}$ and $\mathcal{B}_{n_1, n_2, \rho_{R}}(\mathrsfso{F})$, while Table \ref{T2} gives the approximation error by operators $\mathcal{S}_{n_1, \sigma}^{x}(\mathrsfso{F})$, $\mathcal{S}_{n_2, \sigma}^{y}(\mathrsfso{F})$, $\mathcal{P}_{n_1, n_2, \sigma}(\mathrsfso{F})$ and $\mathcal{B}_{n_1, n_2, \sigma}(\mathrsfso{F})$. It shows that approximation error varies with change in activation function.
	\begin{table}[h] 
		\begin{tabular}{ c c c c c } 
			\hline
			$n_1=n_2$ & $\|\mathcal{S}_{n_1,\rho_{R}}^{x}(\mathrsfso{F}) -\mathrsfso{F}\|$ & $\|\mathcal{S}_{n_2,\rho_{R}}^{y}(\mathrsfso{F}) -\mathrsfso{F}\|$ & $\|\mathcal{P}_{n_1, n_2,\rho_{R}}(\mathrsfso{F}) -\mathrsfso{F}\|$& $\|\mathcal{B}_{n_1, n_2,\rho_{R}}^{y}(\mathrsfso{F}) -\mathrsfso{F}\|$ \\ 
			\hline 
			5   & 0.452497 & 0.156546 & 0.478877  & 1.5438e-15\\  
			15  & 0.101601 & 0.049354 & 0.118881  & 1.6701e-15\\ 
			30  & 0.048032 & 0.023884 & 0.057413  & 1.6787e-15\\  
			50  & 0.027995 & 0.014066 & 0.035111  & 1.5326e-15\\ 
			75  & 0.019564 & 0.009770 & 0.022506  & 1.7282e-15\\ 
			100 & 0.014114 & 0.007005 & 0.015961  & 1.7734e-15\\ 
			\hline 
		\end{tabular}
		\caption{Approximation error of $\mathcal{S}_{n_1,\rho_{R}}^{x}(\mathrsfso{F})$, $\mathcal{S}_{n_2,\rho_{R}}^{y}(\mathrsfso{F})$, $\mathcal{P}_{n_1, n_2,\rho_{R}}(\mathrsfso{F})$, $\mathcal{B}_{n_1, n_2,\rho_{R}}(\mathrsfso{F})$} \label{T1}
	\end{table}
	\begin{table}[h]
		\begin{tabular}{ c c c c c } 
			\hline
			$n_1 =n_2$ & $\|\mathcal{S}_{n_1,\xi}^{x}(\mathrsfso{F}) -\mathrsfso{F}\|$ & $\|\mathcal{S}_{n_2,\xi}^{y}(\mathrsfso{F}) -\mathrsfso{F}\|$ & $\|\mathcal{P}_{n_1, n_2,\xi}(\mathrsfso{F}) -\mathrsfso{F}\|$& $\|\mathcal{B}_{n_1, n_2,\xi}^{y}(\mathrsfso{F}) -\mathrsfso{F}\|$ \\
			\hline
			5    & 0.440815 & 0.114639 & 0.440815  & 1.3971e-15\\  
			15   & 0.054179 & 0.013715 & 0.061273  & 1.5412e-15\\ 
			30   & 0.054179 & 0.003450 & 0.015711  & 1.6072e-15\\  
			50   & 0.004995 & 0.001250 & 0.005642  & 1.3757e-15\\ 
			75   & 0.002218 & 0.000555 & 0.002423  & 1.6474e-15\\ 
			100  & 0.001222 & 0.000289 & 0.001390  & 1.7359e-15\\ 
			\hline
		\end{tabular}
		\caption{Approximation error of $\mathcal{S}_{n_1,\sigma}^{x}(\mathrsfso{F})$, $\mathcal{S}_{n_2,\sigma}^{y}(\mathrsfso{F})$, $\mathcal{P}_{n_1, n_2,\sigma}(\mathrsfso{F})$, $\mathcal{B}_{n_1, n_2,\sigma}(\mathrsfso{F})$} 	\label{T2} 
	\end{table}
	\newline It is evident from Table \ref{T1} and Table \ref{T2} that with an increase in the value of $n_1$ and $n_2$ the  approximation error reduces. The weights and the bias related with $n_1$ and $n_2$ get adjusted to yield the desired results.
	\section{Graphical Analysis}	\label{s7} 
	For graphical analysis, we let $\mathrsfso{F}(x,y) = -\frac{1}{7} \exp \left( - \frac{81}{16} \left( (x - 0.2)^2 + (y - 0.3)^2 \right) \right)$ be the target function. Figure (\ref{fig:1}), (\ref{fig:p}) (\ref{fig:2}), (\ref{fig:3}), (\ref{fig:4}) and (\ref{fig:5}) represent the function $\mathrsfso{F}$, the projection of function $\mathrsfso{F}$,  operators $\mathcal{S}^{x}_{\xi,n_1}\mathrsfso{F}$, $\mathcal{S}^{y}_{\xi,n_2}\mathrsfso{F}$,  bivariate operator $\mathcal{P}_{n_1, n_2}\mathrsfso{F}$ and Generalized Boolean sum  operator $\mathcal{\mathcal{B}}_{n_1, n_2}\mathrsfso{F}$ respectively for $n_1=n_2=15$. It is evident from the graphs that each approximation mimics the graph of the target function $\mathrsfso{F}$ as per theoretical results obtained. 
	
	\begin{figure}
		\begin{subfigure}{.5\textwidth}
			\includegraphics[width=1\linewidth, height=5cm]{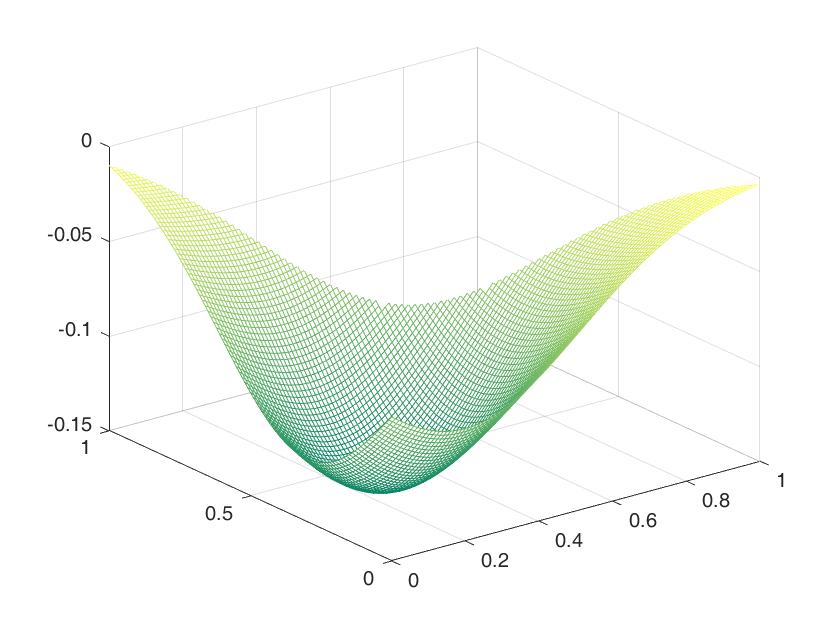}
			\caption{$\mathrsfso{F}(x,y)$}\label{f5}
			\label{fig:1}
		\end{subfigure}
		\begin{subfigure}{.5\textwidth} 
			\includegraphics[width=6.5cm, height=5cm]{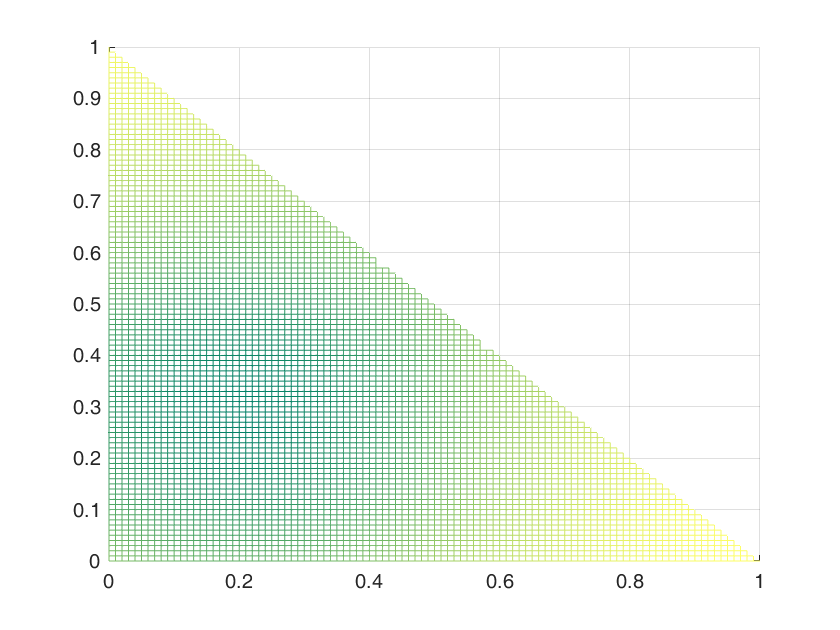} 
			\caption{Projection of function $\mathrsfso{F}$}\label{f5}
			\label{fig:p}
		\end{subfigure}
		\begin{subfigure}{.5\textwidth}
			\includegraphics[width=1\linewidth, height=5cm]{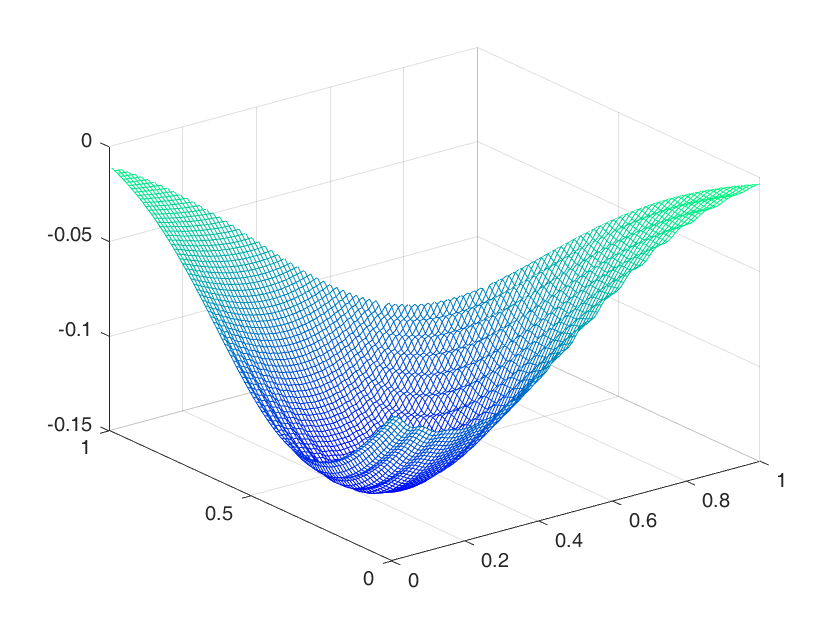}
			\caption{ $\mathcal{S}_{\xi, n}^{x}\mathrsfso{F}$}\label{f1}
			\label{fig:2}
		\end{subfigure}
		\begin{subfigure}{.5\textwidth}
			\includegraphics[width=1\linewidth, height=5cm]{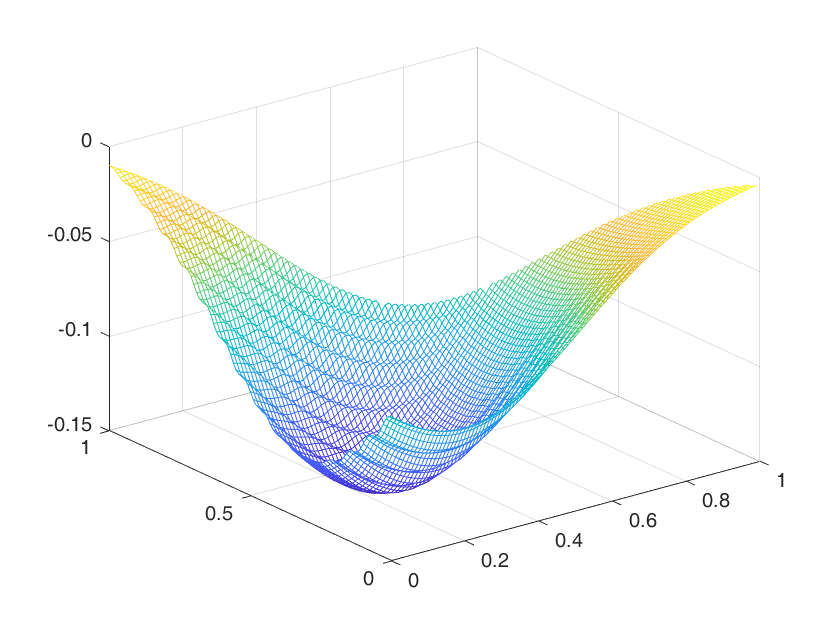}
			\caption{$\mathcal{S}_{\xi, n_2}^{y}\mathrsfso{F}$}\label{f2}
			\label{fig:3}
		\end{subfigure}
		\begin{subfigure}{.5\textwidth}
			\includegraphics[width=1\linewidth, height=5cm]{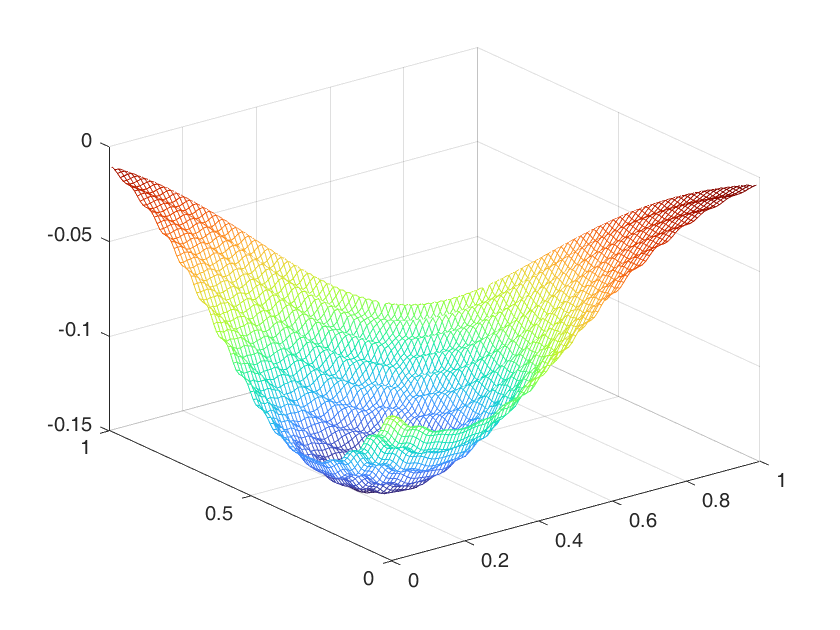}
			\caption{$\mathcal{P}_{n_1 n_2}\mathrsfso{F}$}\label{f3}
			\label{fig:4}
		\end{subfigure}
		\begin{subfigure}{.5\textwidth}
			\includegraphics[width=1\linewidth, height=5cm]{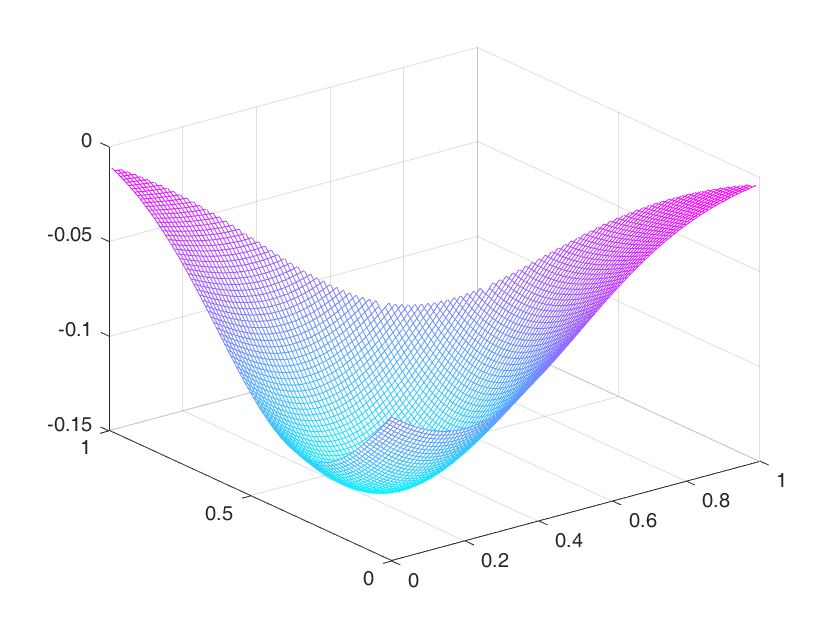}
			\caption{$\mathcal{\mathcal{B}}_{mn}\mathrsfso{F}$}\label{f4}
			\label{fig:5}
		\end{subfigure}
		\caption{}
	\end{figure}
	\newpage
	
	\bibliographystyle{plain}

\end{document}